\documentclass{amsproc}

\usepackage{hyperref}
\usepackage{cite}
\usepackage{amssymb}
\usepackage{amsmath}
\usepackage{amsthm}
\usepackage{amsfonts}
\usepackage{graphicx}
\usepackage{bbm}
\usepackage{xcolor}
\usepackage[toc,page]{appendix}
\usepackage{subfigure}
\usepackage{mathtools}
\usepackage{bm} 
\usepackage[normalem]{ulem}
\usepackage{comment}

\newtheorem{theorem}{Theorem}[section]
\newtheorem{lemma}[theorem]{Lemma}
\newtheorem*{theorem*}{Theorem}
\newtheorem{proposition}[theorem]{Proposition}
\newtheorem{corollary}[theorem]{Corollary}

\theoremstyle{definition}
\newtheorem{definition}[theorem]{Definition}
\newtheorem{example}[theorem]{Example}

\newtheorem{notation}[theorem]{Notation}

\theoremstyle{remark}
\newtheorem{remark}[theorem]{Remark}

\numberwithin{equation}{section}

\newcommand{\R}{\mathbb{R}}  
\newcommand{\Prob}{\mathbb{P}}

\newcommand{\Hei}{\mathbb{H}}

\begin{document}

\title[spectral gap bound H-groups]{Spectral gap bounds on H-type groups}



\author[Carfagnini]{Marco Carfagnini{$^{\dag}$}}
\thanks{\footnotemark {$\dag$} Research was supported in part by an AMS-Simons Travel Grant}
\address{Department of Mathematics\\
University of California, San Diego\\
La Jolla, CA 92093-0112,  U.S.A.}
\email{mcarfagnini@ucsd.edu}

\author[Gordina]{Maria Gordina{$^{\dag  \dag}$}}
\address{ Department of Mathematics\\
University of Connecticut\\
Storrs, CT 06269,  U.S.A.}
\thanks{\footnotemark {$\dag \dag$} Research was supported in part by NSF Grants DMS-2246549 and DMS-1954264.}
\email{maria.gordina@uconn.edu}


\keywords{Sub-Laplacian; H-type group; spectral gap; small ball problem}

\subjclass[2020]{Primary 35P05, 35H10; Secondary 35K08}

\date{\today \ \emph{File:\jobname{.tex}}}

\begin{abstract}
In this note we provide bounds on the spectral gap for the Dirichlet sub-Laplacians on $H$-type groups. We use probabilistic techniques and in particular small deviations of the corresponding hypoelliptic Brownian motion.
\end{abstract}

\maketitle


\tableofcontents

\section{Introduction and main result}

Let $(E, d)$ be a metric space and $\{X_{t}\}_{0 \leqslant t\leqslant T}$ be an $E$-valued stochastic process with continuous paths such that $X_{0} = x_{0}$ a.s. for some $x_{0} \in E$ and $T>0$ is fixed. Denote by $W_{x_{0}} (E)$ the space of $E$-valued continuous functions on $[0,T]$ starting at $x_{0}$. The process $X_{t}$ is said to satisfy a \emph{small deviation principle} with rates $\alpha$ and $\beta$ if there exists a constant $c>0$ such that
\begin{equation}\label{eqn.small.dev.prin.general}
\lim_{\varepsilon\rightarrow 0} - \varepsilon^{\alpha} \vert \log \varepsilon\vert^{\beta} \log \Prob \left( \max_{0 \leqslant t\leqslant T} d (X_{t}, x_{0} )< \varepsilon \right) =c.
\end{equation}
Small deviations  have many applications including metric entropy estimates and limit laws such as Chung's laws of the iterated logarithm. We are interested in such a problem and its analytic consequences when $E$ is a Heisenberg-type (H-type) group.

\begin{notation}\label{notation.evalues}
Let us denote by $\lambda_{1}^{(n)}$ the lowest Dirichlet eigenvalue of $-\frac{1}{2} \Delta_{\R^{n}}$ on the unit ball in $\R^{n}$.
\end{notation}

Suppose $B_{t}$ is a standard Brownian motion on $\R^{n}$, then 
\begin{equation}\label{eqn.small.dev.BM}
\lim_{\varepsilon\rightarrow 0} - \varepsilon^{2} \log \Prob \left( \max_{0 \leqslant t\leqslant T} \vert B_{t} \vert < \varepsilon \right) = \lambda_{1}^{(n)} T
\end{equation}
connects the small deviation principle and the spectral gap for the infinitesimal generator of $B_{t}$. In general, small deviations for a Markov process $X_{t}$ do not necessarily encode information about the spectral gap for the infinitesimal generator of $X_{t}$. Indeed, in \cite{Carfagnini2022} it was shown that the Kolmogorov diffusion $K_{t}$ satisfies \eqref{eqn.small.dev.prin.general} with $c= \lambda_{1}^{(1)}T$, though $\lambda_{1}^{(1)}$ is not the spectral gap for the infinitesimal generator of $K_{t}$.

When $X_{t}$ is a Markov diffusion whose infinitesimal generator is a sub-Laplacian $\Delta_{\mathbb{G}}$ on a Carnot group $\mathbb{G}$, then the small deviation principle \eqref{eqn.small.dev.BM} for $X_{t}$ and the spectral gap of $\Delta_{\mathbb{G}}$ holds true. Recall that a homogeneous Carnot group $\mathbb{G}$ is $\R^{N}$ endowed with a non-commutative group structure. Such a group is a sub-Riemannian manifold with  the corresponding sub-Laplacian $\Delta_{\mathbb{G}}$. Let $X_{t}$ be the $\mathbb{G}$-valued stochastic process whose infinitesimal generator is the sub-Laplacian $\frac{1}{2}\Delta_{\mathbb{G}}$, and let $\vert \cdot \vert$ be a homogeneous norm on $\mathbb{G}$. In  \cite[Theorem 3.4]{CarfagniniGordina2023} it was shown that the Dirichlet sub-Laplacian $-\frac{1}{2} \Delta_{\mathbb{G}}$ restricted to a bounded open connected subset $\Omega$ of $\mathbb{G}$ has a spectral gap, moreover, its spectrum is discrete, non-negative, and the lowest eigenvalue is strictly positive. By \cite[Corollary 5.4]{CarfagniniGordina2023} it then follows that

\begin{equation}\label{eqn.small.dev.hypo.BM}
\lim_{\varepsilon\rightarrow 0} - \varepsilon^{2} \log \Prob \left( \max_{0 \leqslant t\leqslant T} \vert X_{t} \vert < \varepsilon \right) = \lambda_{1} T,
\end{equation}
where $\lambda_{1}$ is the spectral gap of the sub-Laplacian $-\frac{1}{2}\Delta_{\mathbb{G}}$ restricted to the ball $\{ x\in \mathbb{G} : \vert x \vert < 1 \}$.

In this paper we use the small deviations principle \eqref{eqn.small.dev.hypo.BM} to obtain bounds on the spectral gap of sub-Laplacians on $H$-type groups. This is a class of Carnot groups and we refer to Section~\ref{sec.groups} for more details on $H$-type groups. The main result is the following theorem.

\begin{theorem*}[Theorem \ref{thm.main}]
Let $\mathfrak{g}$ be an $H$-type Lie algebra with center $\mathfrak{z} \cong \R^{n}$. Let $\mathbb{G} \cong \R^{m} \times \R^{n}$ be the corresponding $H$-type group  with the homogeneous norm $\vert \cdot \vert$ and the sub-Laplacian $\Delta_{\mathbb{G}}$. Then
\begin{equation*}
\lambda_{1}^{(m)} \leqslant \lambda_{1} \leqslant c \left( \lambda_{1}^{(m)},\lambda_{1}^{(n)}  \right),
\end{equation*}
where $\lambda_{1} = \lambda_{1} (m,n)$ is the spectral gap of $-\frac{1}{2} \Delta_{\mathbb{G}}$ restricted to the unit ball $\{ x\in\mathbb{G}: \vert x \vert <1\}$, and
\begin{align*}
& c \left( \lambda_{1}^{(m)}, \lambda_{1}^{(n)}  \right) := f( x^{\ast} )= \inf_{0 < x<1} f(x),
\\
& f(x) = \frac{\lambda_{1}^{(m)}}{\sqrt{1-x}} + \frac{\lambda_{1}^{(n)} \sqrt{1-x}}{4x},
\\
& x^{\ast}= \frac{\sqrt{ \left( \lambda_{1}^{(n)}\right)^{2} + 32 \lambda_{1}^{(n)} \lambda_{1}^{(m)} } - 3\lambda_{1}^{(n)}  }{2 \left( 4\lambda_{1}^{(m)} - \lambda_{1}^{(n)} \right)},
\end{align*}
where $\lambda_{1}^{(n)}$ is the lowest Dirichlet eigenvalues in the unit ball in $\R^{n}$.
\end{theorem*}

The proof of Theorem \ref{thm.main} is of probabilistic nature and it relies on Lemma \ref{lemma.IW.general} which is a version of the following well-known fact about the stochastic L\'evy area. If $B_{t}$ is a two-dimensional Brownian motion and
\[
A_{t} := \frac{1}{2} \int_{0}^{t} B_{1}(s) dB_{2}(s) - B_{2}(s)dB_{1}(s)
\]
the corresponding L\'evy area, then a classical result in stochastic analysis \cite[Ch.VI Example 6.1]{IkedaWatanabe1989} states that there exists a standard real-valued Brownian motion $W_{t}$ independent of $\vert B_{t}\vert$ such that

\begin{align}\label{eqn.idk}
A_{t} = W_{\tau (t)},
\end{align}
where $\tau (t) := \frac{1}{4} \int_{0}^{t} \vert B_{s} \vert^{2} ds$. In Lemma \ref{lemma.IW.general} we prove a version of \eqref{eqn.idk}, where $B_{t}$ is replaced by an $m$-dimensional Brownian motion, and $A_{t}$ by an $n$-dimensional martingale consisting of stochastic integrals depending on $B_{t}$.

\section{Geometric background}\label{sec.groups}

\subsection{Carnot groups}
We begin by recalling basic facts about Carnot groups.

\begin{definition}[Carnot groups]
We say that $\mathbb{G}$ is a Carnot group of step $r$ if $\mathbb{G}$  is a connected and simply connected Lie group whose Lie algebra $\mathfrak{g}$ is \emph{stratified}, that is, it can be written as
\begin{equation*}
\mathfrak{g}=V_{1}\oplus\cdots\oplus V_{r},
\end{equation*}
where
\begin{align*}
& \left[V_{1}, V_{i-1}\right]=V_{i}, \hskip0.1in 2 \leqslant i \leqslant r,
\\
& [ V_{1}, V_{r} ]=\left\{ 0 \right\}.
\end{align*}
\end{definition}
To avoid degenerate cases we assume that the dimension of the Lie algebra $\mathfrak{g}$ is at least $3$. In addition we will use a stratification such that the center of $\mathfrak{g}$ is contained in  $V_{r}$. We generally assume that $r \geqslant 2$ to exclude the case when the corresponding Laplacian is elliptic. In particular, Carnot groups are nilpotent. We will use $\mathcal{H}:=V_{1}$ to denote the space of \emph{horizontal} vectors that generate the rest of the Lie algebra, noting that $V_2=[\mathcal{H}, \mathcal{H}], ..., V_r = \mathcal{H}^{(r)}$.  Finally, by \cite[Proposition 2.2.17, Proposition 2.2.18]{BonfiglioliLanconelliUguzzoniBook} we can assume without loss of generality that a Carnot group can be identified with a \emph{homogeneous Carnot group}. For $i=1, ..., r$, let $d_{i}=\operatorname{dim} V_{i}$ and $d_0=0$. The Euclidean space underlying $G$ has  dimension
\begin{align*}
N:=\sum_{i=1}^{r} d_{i},
\end{align*}
that is, $\mathbb{G}\cong \R^{N}$, and the homogeneous dimension of $\mathbb{G}$ is given by
\begin{equation*}
Q:=\sum_{i=1}^{r} i \cdot  d_{i}.
\end{equation*}
A homogeneous Carnot group is equipped with a natural family of \emph{dilations} defined for any $a>0$ by
\begin{equation*}
D_{a}\left( x_{1}, \dots, x_{N} \right):=\left( a^{\sigma_{1}}x_{1}, \dots, a^{\sigma_{N}}x_{N}\right),
\end{equation*}
where $\sigma_{j} \in \mathbb{N}$ is called the \emph{homogeneity} of $x_{j}$, with
\[
\sigma_{j}= i \quad\text{ for } \sum_{k=0}^{i-1}d_{k} +1 \leqslant j \leqslant \sum_{k=1}^{i}d_{k},
\]
with $i=1,\ldots,r$ and recalling that $d_0=0$.
That is, $\sigma_{1}=\dots= \sigma_{d_{1}}=1, \sigma_{d_{1}+1}=\cdots=\sigma_{d_1+d_2}=2$, and so on.

We assume that $\mathcal{H}$ is equipped with an inner product $\langle \cdot, \cdot\rangle_{\mathcal{H}}$, in which case the Carnot group has a natural sub-Riemannian structure. Namely, one may use left translation to define a \emph{horizontal distribution} $\mathcal{D}$ as a sub-bundle of the tangent bundle $T\mathbb{G}$, and a metric on $\mathcal{D}$. First, we identify the space $\mathcal{H} \subset \mathfrak{g}$ with $\mathcal{D}_{e}\subset T_e\mathbb{G}$. Then for $g\in \mathbb{G}$ let $L_{g}$ denote left translation $L_{g}h =gh$, and define $\mathcal{D}_{g}:=(L_g)_{\ast}\mathcal{D}_{e}$ for any $g \in \mathbb{G}$. A metric on $\mathcal{D}$ may then be defined by
\begin{align*}
\langle u, v \rangle_{\mathcal{D}_g} &:= \langle  (L_{g^{-1}})_{\ast} u, (L_{g^{-1}})_{\ast}v\rangle_{\mathcal{D}_e}
\\
&= \langle (L_{g^{-1}})_{\ast}u,(L_{g^{-1}})_{\ast}v\rangle_{\mathcal{H}} \qquad \text{for all } u, v \in\mathcal{D}_g.
\end{align*}
We will sometimes identify the horizontal distribution $\mathcal{D}$ and $\mathcal{H}$. Vectors in $\mathcal{D}$ are called \emph{horizontal}.

Let $\{ X_{1}, \ldots, X_{d_{1}} \}$ be an orthonormal basis for $\mathcal{H}$ of left-invariant vector fields, then the sum of squares operator
\[
\Delta_{\mathbb{G}} := \sum_{j=1}^{d_{1}} X_{j}^{2}
\]
is called the canonical sub-Lalplacian on $\mathbb{G}$.  By \cite[Section 3]{DriverGrossSaloff-Coste2009a} the operator $\Delta_{\mathbb{G}}$ depends only on the inner product on $\mathcal{H}$, and not on the choice of the basis.

\begin{definition}\label{dfn.hom.norm}
Suppose $\mathbb{G} = \left( \R^N, \star, \delta_\lambda \right)$ is a homogeneous Carnot group, and $\rho:\mathbb{G} \rightarrow [0,\infty)$ is a continuous function with respect to the Euclidean topology. Then $\rho$ is a \emph{homogeneous norm} if it satisfies the following properties
\begin{align*}
& \rho\left( \delta_{\lambda} (x) \right) = \lambda \rho(x)  \text{ for every }  \lambda >0   \text{ and } x \in \mathbb{G},
\\
& \rho(x) >0  \text{ if and only if }  x\not=0.
\end{align*}
The norm $\rho$ is called \emph{symmetric} if it satisfies $\rho\left(x^{-1} \right) = \rho\left( x \right)$ for every $x \in \mathbb{G}$.
\end{definition}

\begin{definition}\label{dfn.hypo.BM}
A $\mathbb{G}$-valued Markov process $g_{t}$ is called a hypoelliptic (or horizontal) Brownian motion if its infinitesimal generator is $\frac{1}{2} \Delta_{\mathbb{G}}$.
\end{definition}

For $g\in \mathbb{G}$, let us denote by $\theta^{\ell}_{g}$ the \emph{left Maurer-Cartan form} on $\mathbb{G}$, that is, a $\mathfrak{g}$-valued 1-form on $\mathbb{G}$ defined by $\theta^{\ell}_{g} (v) := d L_{g} (v)$, $v\in T_{g} \mathbb{G}$. The horizontal Brownian motion $g_{t}$ is then the solution to the $\mathfrak{g}$-valued stochastic differential equation
\begin{align*}
& \theta^{\ell}_{g}  (dg_{t}) = \left( dB_{t}, 0 \right),
\\
& g_{0} = e,
\end{align*}
where $B_{t}$ is an $\mathcal{H}$-valued Brownian motion and $e$ is the identity in $\mathbb{G}$.

\subsection{$H$-type groups}

H-type or Heisenberg-type groups are examples of Carnot groups. They were first introduced by Kaplan in \cite{Kaplan1980} and  basic facts about these groups can be found in \cite[Chapter 18]{BonfiglioliLanconelliUguzzoniBook}.

\begin{definition}\label{def.H.group}
A finite dimensional real Lie algebra  $\mathfrak{g}$ is said to be an \emph{$H$-type Lie algebra} if there exists an inner product $\langle \cdot, \cdot \rangle$ on $\mathfrak{g}$ such that
\[
[ \mathfrak{z}^{\perp}, \mathfrak{z}^{\perp}] = \mathfrak{z},
\]
where $\mathfrak{z}$ is the center of $\mathfrak{g}$, and for every fixed $Z\in \mathfrak{z}$, the map $J_{Z} : \mathfrak{z}^{\perp} \rightarrow \mathfrak{z}^{\perp}$ defined by
\[
\langle J_{Z} X, Y \rangle = \langle Z, [X,Y] \rangle
\]
is an orthogonal map whenever $\langle Z, Z \rangle =1$.

An \emph{$H$-type group} $\mathbb{G}$ is a connected and simply connected Lie group whose Lie algebra $\mathfrak{g}$ is an $H$-type algebra.
\end{definition}

\begin{remark}[Corollary 1 in \cite{Kaplan1980}]
Let $n$ and $m$ be two integers. Then there exists an $H$-type Lie algebra of dimension $m+n$ whose center has dimension $n$ if and only if $n <\rho (m)$, where $\rho$ is the Hurwitz-Radon function
\[
\rho : \mathbb{N} \rightarrow \mathbb{N}, \; \rho (n) = 8p+q, \; \text{where} \; n = \text{(odd)} \cdot 2^{4p+q}, \; 0\leqslant q\leqslant 3.
\]
\end{remark}

\begin{example}\label{ex.Heisenberg}\rm
The Heisenberg group $\Hei$ is the Lie group identified with $\R^{3}$ endowed with the following group multiplication law
\begin{align*}
& \left( \mathbf{v}_{1}, z_{1} \right) \cdot \left( \mathbf{v}_{2}, z_{2} \right) := \left( x_{1}+x_{2}, y_{1}+y_{2}, z_{1}+z_{2} + \frac{1}{2}\omega\left( \mathbf{v}_{1}, \mathbf{v}_{2} \right)\right),
\\
& \text{ where } \mathbf{v}_{1}=\left( x_{1}, y_{1} \right), \mathbf{v}_{2}=\left( x_{2}, y_{2} \right) \in \mathbb{R}^{2},
\\
& \omega: \mathbb{R}^{2} \times \mathbb{R}^{2} \longrightarrow \mathbb{R}, \; \; \omega\left( \mathbf{v}_{1}, \mathbf{v}_{2} \right):= x_{1}y_{2}-x_{2} y_{1}
\end{align*}
is the standard symplectic form on $\mathbb{R}^{2}$. The identity in $\Hei$ is $e=(0, 0, 0)$ and the inverse is given by $\left( \mathbf{v}, z \right)^{-1}= (-\mathbf{v},-z)$.

The Heisenberg Lie algebra $\mathfrak{h}$ is spanned by the vector fields
\[
X=\frac{\partial}{\partial x} - \frac{y}{2}  \frac{\partial}{\partial z}, \; \; Y= \frac{\partial}{\partial y} + \frac{x}{2}  \frac{\partial}{\partial z}, \; \; Z= \frac{\partial}{\partial z}.
\]
Note that $[X,Y]= Z$ is the only non-zero Lie bracket, and hence $\mathfrak{h}$ is an $H$-type Lie algebra under the inner product such that $\left\{ X, Y, Z\right\}$ is an orthonormal basis. In particular, $\mathfrak{z}=\operatorname{span} \left\{ Z \right\}$, and $J_{Z} X = Y, J_{Z} Y = -X$.
\end{example}

\begin{example}\rm
For $n\geqslant 1$, the Heisenberg-Weyl group $\Hei_{n}$ is the Lie group identified with $\R^{2n+1}$ with the following group multiplication law
\begin{align*}
& \left( \mathbf{v}, z \right) \cdot \left( \mathbf{v}^{\prime}, z^{\prime} \right) := \left( \mathbf{v} +  \mathbf{v}^{\prime}, z+z^{\prime} + \frac{1}{2}\omega\left( \mathbf{v}, \mathbf{v}^{\prime} \right)\right),
\\
& \text{ where } \mathbf{v}= (x_{1}, \ldots x_{n}),\;\mathbf{v}^{\prime}=(x_{1}^{\prime}, \ldots, x_{2n}^{\prime}) \in \mathbb{R}^{2n}, \; z, z^{\prime} \in \R,
\\
& \omega: \mathbb{R}^{2n} \times \mathbb{R}^{2n} \longrightarrow \mathbb{R}, \; \;  \omega\left( \mathbf{v}_{1}, \mathbf{v}_{2} \right):= \sum_{j=1}^{n} x_{2j-1}x_{2j}^{\prime} - x_{2j-1}^{\prime}x_{2j}.
\end{align*}
Its Lie algebra $\mathfrak{h}_{n}$ is generated by
\[
X_{2j-1}=\frac{\partial}{\partial x_{2j-1}} - \frac{1}{2} x_{2j} \frac{\partial}{\partial z}, \; \; X_{2j}= \frac{\partial}{\partial x_{2j}} + \frac{1}{2} x_{2j-1} \frac{\partial}{\partial z}, \; \; Z= \frac{\partial}{\partial z},
\]
where $j=1,\ldots n$. Note that the only non zero brackets are $[X_{2j-1}, X_{2j}]=Z$, and $\mathfrak{z}=\operatorname{span}\left\{ Z \right\}$, and hence $\mathfrak{h}_{n}$ is an $H$-type algebra under the inner product that makes $\left\{ X_{1}, \ldots, X_{2n}, Z \right\}$ orthonormal. Moreover, one has that $J_{Z}X_{2j-1} = X_{2j}$, $J_{Z} X_{2j} = - X_{2j-1}$ for all $j=0,\ldots n$.
\end{example}

\begin{remark}\label{rmk.H.are.Carnot}
Let $\mathbb{G}$ be an $H$-type group, then $\mathbb{G}$ is a Carnot group of step $2$. Indeed, if $\mathfrak{z}$ denotes the center of its Lie algebra $\mathfrak{g}$, then one can consider the stratification $V_{1} := \mathfrak{z}^{\perp}, \, V_{2} := \mathfrak{z}$.
\end{remark}

The following characterization of $H$-type groups can be found in \cite[Theorem 18.2.1]{BonfiglioliLanconelliUguzzoniBook}.

\begin{theorem}\label{thm.char.H.groups}
Let $\mathbb{G}$ be an $H$-type group. Then there exist integers $m, n$ such that $\mathbb{G}$ is isomorphic to $\R^{m+n}$ with the group law given by
\begin{equation}\label{eqn.group.law}
x \cdot y = \left( \overline{x} + \overline{y}, \widehat{x}_{1} + \widehat{y}_{1} + \frac{1}{2} \langle U^{(1)}\overline{x}, \overline{y}\rangle, \ldots,  \widehat{x}_{n} + \widehat{y}_{n} + \frac{1}{2} \langle U^{(n)}\overline{x}, \overline{y}\rangle  \right),
\end{equation}
for any $x := (\overline{x}, \widehat{x}), \, y = (\overline{y}, \widehat{y} ) \in \R^{m+n}$, where the matrices $U^{(1)}, \ldots, U^{(n)}$ satisfy the following properties

1. $U^{(i)}$ is an $m \times m$ skew-symmetric and orthogonal matrix for $i=1, \ldots, n$.

2. $U^{(i)} U^{(j)} + U^{(j)}U^{(i)} =0$ for every $i\neq j$.
\end{theorem}

\begin{remark}
The integers $m$ and $n$ in Theorem \ref{thm.char.H.groups} only depend on the structure of $\mathbb{G}$. More precisely, $m= \dim \mathfrak{z}^{\perp}$ and $n= \dim \mathfrak{z}$.
\end{remark}

By \cite[Proposition~5.1.4, p.~230]{BonfiglioliLanconelliUguzzoniBook} one has that all homogeneous norms on a Carnot group are equivalent. $H$-type groups are Carnot groups of step $2$ and hence, without loss of generality, we can focus our attention on the homogeneous norm
\begin{equation}\label{eqn.hom.norm.H.typw}
\vert x \vert := \left( \vert \overline{x}\vert_{\R^{m}}^{4} + \vert \widehat{x} \vert_{\R^{n}}^{2} \right)^{\frac{1}{4}}
\end{equation}
for any $x= \left( \overline{x}, \widehat{x} \right) \in \R^{m+n}$. We refer to Remark 
\ref{rmk.any.norm} for bounds on the spectral gap on balls with respect to a general homogeneous norm. 

\begin{remark}\label{rmk.properties.matrices}
The matrices $U^{(i)}$ in Theorem \ref{thm.char.H.groups} satisfy
\begin{equation}\label{eqn.prop.matrices}
\langle U^{(i)} \overline{x}, U^{(j)} \overline{x} \rangle =0,
\end{equation}
for any $\overline{x}\in \R^{m}$ and any $i\neq j$ in $\{ 1,\ldots, n\}$. Indeed, $U^{(i)} U^{(i) \, T} = I$ and  $U^{(i)} = - U^{(i) \, T}$, and hence, for any $i\neq j$
\begin{align*}
& \langle U^{(i)} \overline{x}, U^{(j)} \overline{x} \rangle = \langle U^{(j)} U^{(i)} \overline{x},  - \overline{x} \rangle = \langle U^{(i)} U^{(j)} \overline{x},   \overline{x} \rangle.
\end{align*}
On the other hand
\begin{align*}
&  \langle U^{(i)} \overline{x}, U^{(j)} \overline{x} \rangle =  \langle -\overline{x}, U^{(i)}  U^{(j)} \overline{x} \rangle =- \langle U^{(i)}  U^{(j)}\overline{x},  \overline{x} \rangle .
\end{align*}
\end{remark}

We now describe the Maurer-Cartan form, the hypoelliptic Brownian motion on $\mathbb{G}$, and its infinitesimal generator $\frac{1}{2}\Delta_{\mathbb{G}}$.

\begin{proposition}\label{prop.MC.for.Brownian.motion}
Let $\mathbb{G}$ be an $H$-type group. Then the Maurer-Cartan form is given by

\begin{equation*}
\theta^{\ell}_{k} (v)=  \left( \overline{v},
\widehat{v}_{1} - \frac{1}{2} \langle U^{(1)}\overline{k}, \overline{v}\rangle, \ldots, \widehat{v}_{n} - \frac{1}{2} \langle U^{(n)}\overline{k}, \overline{v}\rangle \right)
\end{equation*}
for any $k=( \overline{k}, \widehat{k} )\in \mathbb{G}$ and $v= (\overline{v}, \widehat{v})\in T_{k} \mathbb{G}$.

Left-invariant vector fields at $x= (\overline{x}, \widehat{x})$ can be written as
\begin{align*}
& X_{j}= \frac{\partial}{\partial \overline{x}_{j}}-\frac{1}{2}\sum_{s=1}^{n} \left( \sum_{i=1}^{m} U^{(s)}_{ji} \overline{x}_{i} \right)\frac{\partial}{\partial \widehat{x}_{s}}, & j=1, \ldots, m,
\\
& Z_{i}= \frac{\partial}{\partial \widehat{x}_{i}}, & i=1, \ldots, n.
\end{align*}
A hypoelliptic Brownian motion on $\mathbb{G}$ starting at the identity can be written as
\begin{equation*}
g_{t} = \left( B_{t}, A_{t} \right),
\end{equation*}
where $B_{t}$ is a standard Brownian motion on $\R^{m}$ and $A_{t} = ( A_{1} (t), \ldots, A_{n} (t) )$, with
\begin{equation}\label{eqn.hypo.bm.coefficients}
A_{i}(t) = \frac{1}{2} \int_{0}^{t} \langle U^{(i)} B_{s}, dB_{s} \rangle, \; i=1,\ldots, n.
\end{equation}
\end{proposition}

\begin{proof}
Let $\gamma (t)= (\overline{\gamma} (t), \widehat{\gamma} (t) )$, $0\leqslant t \leqslant 1$, be a curve in $\mathbb{G}$ with $\gamma(0)=k \in \mathbb{G}$ and $\gamma^{\prime} (0) = v  \in T_{k}\mathbb{G}$. Then
\begin{align*}
& \theta^{\ell}_{k} (v):= dL_{k}(v) = \left.\frac{d}{dt}\right|_{t=0} k^{-1} \gamma (t)
\\
& = \left.\frac{d}{dt}\right|_{t=0} \left( -\overline{k} + \overline{\gamma}(t), -\widehat{k}_{1} + \widehat{\gamma}_{1}(t) - \frac{1}{2} \langle U^{(1)}\overline{k}, \overline{\gamma}(t)\rangle, \ldots,  \right.
\\
&
\left.
-\widehat{k}_{n} + \widehat{\gamma}_{n}(t) - \frac{1}{2} \langle U^{(n)}\overline{k}, \overline{\gamma}(t)\rangle  \right)
\\
& = \left(\overline{v}, \widehat{v}_{1} - \frac{1}{2} \langle U^{(1)}\overline{k}, \overline{v}\rangle, \ldots, \widehat{v}_{n} - \frac{1}{2} \langle U^{(n)}\overline{k}, \overline{v}\rangle \right),
\end{align*}
from which the expression for left-invariant vector fields follows.

Let $g_{t}$ be a hypoelliptic Brownian motion on $\mathbb{G}$, then
\begin{align*}
& (dB_{t}, 0)= \theta^{\ell}_{g_{t}} (dg_{t})
\\
& =   \left( d\overline{g}_{t},d\widehat{g}_{1}(t) - \frac{1}{2} \langle U^{(1)}\overline{g}_{t}, d\overline{g}_{t}\rangle, \ldots, d \widehat{g}_{n}(t) - \frac{1}{2} \langle U^{(n)}\overline{g}_{t}, d\overline{g}_{t}\rangle \right),
\end{align*}
that is,
\[
g_{t} = \left( B_{t},  \frac{1}{2}\int_{0}^{t} \langle U^{(1)}B_{s}, dB_{s}\rangle, \ldots,  \frac{1}{2} \int_{0}^{t}\langle U^{(n)}B_{s}, dB_{s}\rangle  \right),
\]
where $B_{t}$ is a standard Brownian motion on $\R^{m}$.
\end{proof}

\begin{corollary}\label{cor.sub.lap}
The sub-Laplacian on an $H$-type group $\mathbb{G}$ is given by
\begin{equation*}
\Delta_{\mathbb{G}} = \sum_{j=1}^{m} X_{j}^{2}= \Delta_{ \overline{x}}  - \sum_{i=1}^{n} \langle U^{(i)} \overline{x}, \nabla_{\overline{x}} \rangle \frac{\partial}{\partial \widehat{x}_{i}} + \frac{1}{4}  \vert \overline{x}\vert^{2} \Delta_{\widehat{x}},
\end{equation*}
where $x= ( \overline{x}, \widehat{x} )$,
\begin{align*}
& \Delta_{ \overline{x}} =\sum_{j=1}^{m} \frac{\partial^{2}}{\partial \overline{x}^{2}_{j}},
\\
& \Delta_{ \widehat{x}} =\sum_{i=1}^{n} \frac{\partial^{2}}{\partial \widehat{x}^{2}_{i}},
\\
& \text{and } \; \nabla_{\overline{x}} = \left(\frac{\partial}{\partial \overline{x}_{1}}, \ldots, \frac{\partial}{\partial \overline{x}_{m}} \right).
\end{align*}
is the horizontal gradient.
\end{corollary}
\begin{proof}
One has that
\begin{align*}
&\sum_{j=1}^{m} X_{j}^{2} = \sum_{j=1}^{m} \left( \partial_{\overline{x}_{j}} - \frac{1}{2} \sum_{s=1}^{n} \sum_{i=1}^{m} \left( U_{ji}^{(s)} \overline{x}_{i} \right) \partial_{\widehat{x}_{s}} \right)^{2}
\\
& =  \Delta_{ \overline{x}}  - \frac{1}{2} \sum_{j,i=1}^{m}\sum_{s=1}^{n} \left( \partial_{\overline{x}_{j}} \left( U^{(s)}_{ji} \overline{x}_{i} \partial_{\widehat{x}_{s}} \right) + U^{(s)}_{ji} \overline{x}_{i} \partial^{2}_{\widehat{x}_{s} \overline{x}_{j}} \right)
\\
& + \frac{1}{4} \sum_{p,s=1}^{n} \sum_{j,i =1}^{m} U^{(s)}_{ji} \overline{x}_{i} U^{(p)}_{jl} \overline{x}_{l} \partial^{2}_{\widehat{x}_{s} \widehat{x}_{p}}
\\
& =  \Delta_{ \overline{x}} -  \frac{1}{2} \sum_{j,i=1}^{m}\sum_{s=1}^{n} \left( U^{(s)}_{ji} \delta_{ji} \partial_{\widehat{x}_{s}} + 2U^{(s)}_{ji} \overline{x}_{i} \partial^{2}_{\widehat{x}_{s} \overline{x}_{j}} \right) + \frac{1}{4} \sum_{s,p=1}^{n}\langle U^{(s)} \overline{x}, U^{(p)} \overline{x} \rangle \partial^{2}_{\widehat{x}_{s} \widehat{x}_{p}},
\end{align*}
and by $\eqref{eqn.prop.matrices}$
\begin{align*}
&  \Delta_{ \overline{x}} -  \frac{1}{2}\sum_{s=1}^{n} \left(  \sum_{i=1}^{m}U^{(s)}_{ii} \partial_{\widehat{x}_{s}} + 2 \sum_{j,i=1}^{m}U^{(s)}_{ji} \overline{x}_{i} \partial^{2}_{\widehat{x}_{s} \overline{x}_{j}} \right) + \frac{1}{4} \sum_{s,p=1}^{n}\langle U^{(s)} \overline{x}, U^{(p)} \overline{x} \rangle \partial^{2}_{\widehat{x}_{s} \widehat{x}_{p}}
\\
& = \Delta_{ \overline{x}} - \sum_{s=1}^{n} \sum_{j,i=1}^{m}U^{(s)}_{ji} \overline{x}_{i} \partial^{2}_{\widehat{x}_{s} \overline{x}_{j}} + \frac{1}{4} \sum_{s=1}^{n} \vert U^{(s)} \overline{x}\vert^{2} \partial^{2}_{\widehat{x}_{s}^{2}}
\\
& =  \Delta_{ \overline{x}} - \sum_{s=1}^{n} \langle U^{(s)} \overline{x}, \nabla_{\overline{x}} \rangle \frac{\partial}{\partial \widehat{x}_{s}} + \frac{1}{4} \vert \overline{x}\vert^{2} \Delta_{\widehat{x}},
\end{align*}
where in the second to last line we used that $U^{(s)}_{ii}=0$ for all $i=1,\ldots, m$ and all $s=1, \ldots, n$.
\end{proof}

\section{Proof of the main result}

\subsection{A probabilistic lemma}

Let us recall that  if $B_{t}$ is a two-dimensional standard Brownian motion and $A_{t} := \frac{1}{2} \int_{0}^{t} B_{1}(s) dB_{2}(s) - B_{2}(s)dB_{1}(s)$ the corresponding L\'evy area, then $A_{t} = W_{\tau (t)}$, where $W_{t}$ is a standard real-valued Brownian motion independent of $\vert B_{t} \vert$, \cite[p.470]{IkedaWatanabe1989}.  This idea has been used in \cite{CarfagniniGordina2022} to study small deviations for a hypoelliptic Brownian $g_{t}$ motion on the Heisenberg group $\Hei$. The next lemma extends the result  in \cite{IkedaWatanabe1989} to the vector-valued case.

\begin{lemma}\label{lemma.IW.general}
Let  $g_{t} = (B_{t}, A_{t})$ be a hypoelliptic Brownian motion on an $H$-type group $\mathbb{G}  \cong \R^{m+n}$, where $B_{t}$ is a standard Brownian motion in $\R^{m}$ and $A_{t} = ( A_{1} (t), \ldots, A_{n} (t) )$, with
\[
A_{i}(t) = \frac{1}{2} \int_{0}^{t} \langle U^{(i)} B_{s}, dB_{s} \rangle, \; i=1,\ldots, n.
\]
\end{lemma}
Then there exists an $n$-dimensional Brownian motion $W_{t}$ such that

1. $\vert W_{t} \vert$ is independent of $\vert B_{t} \vert$.

2. $A_{t} = W_{\tau (t)}$, where $\tau (t) := \frac{1}{4} \int_{0}^{t} \vert B_{s} \vert^{2} ds$.

\begin{proof}
Let $X_{t}$ be the one-dimensional Brownian motion given by
\[
X_{t} = \sum_{k=1}^{m} \int_{0}^{t} \frac{B_{k}(s)}{\vert B_{s} \vert} dB_{s},
\]
and let us consider the $(n+1)$-dimensional martingale $(X_{t}, A_{1}(t), \ldots, A_{n} (t) )$. Note that the quadratic variation of $A_{i}(t)$ is independent of $i$ since the matrices $U^{(i)}$ are orthogonal for $i=1,\ldots, n$. Indeed,
\begin{align*}
dA_{i}(t) =  \frac{1}{2} \langle U^{(i)} B_{t}, dB_{t} \rangle,
\end{align*}
from which it follows that
\begin{align*}
d \langle A_{i} \rangle_{t} = \frac{1}{4} \sum_{k=1}^{m} \left(U^{(i)} B_{t} \right)^{2}_{k}dt = \frac{1}{4} \vert U^{(i)} B_{t} \vert^{2} dt= \frac{1}{4} \vert  B_{t} \vert^{2} dt.
\end{align*}

 We claim that the following covariations are zero
\begin{align}
& \langle A_{i}, A_{j} \rangle_{t} =0, \; \text{for all} \, i \neq j \, \text{in} \, \{ 1,\ldots, n\} \label{eqn.cov.one}
\\
&
\langle A_{i}, X \rangle_{t} =0, \; \text{for all} \, i=1,\ldots n.\label{eqn.cov.two}.
\end{align}
Indeed, by \eqref{eqn.hypo.bm.coefficients}
\begin{align*}
& d (A_{i} (t) + A_{j} (t) ) = \frac{1}{2} \langle U^{(i)} B_{t} + U^{(j)} B_{t}, dB_{t} \rangle,
\end{align*}
and hence, by \eqref{eqn.prop.matrices}
\begin{align*}
& d \langle A_{i} + A_{j} \rangle_{t} = \frac{1}{4} \sum_{k=1}^{m} \left( U^{(i)} B_{t} + U^{(j)} B_{t} \right)^{2}_{k} dt
\\
& =  \frac{1}{4} \sum_{k=1}^{m}  \left( U^{(i)} B_{t}\right)^{2}_{k}dt + \frac{1}{4} \sum_{k=1}^{m}  \left( U^{(j)} B_{t}\right)^{2}_{k}dt + \frac{1}{2} \langle   U^{(i)} B_{t},  U^{(j)} B_{t} \rangle dt
\\
& = d \langle A_{i} \rangle_{t} + d \langle A_{j} \rangle_{t},
\end{align*}
for all $i\neq j$ in $1, \ldots, n$, which proves \eqref{eqn.cov.one}. Similarly,
\begin{align*}
& d (X_{t} + A_{i} (t) ) = \sum_{k=1}^{m} \left( \frac{1}{\vert B_{t} \vert} B_{k}(t) + \frac{1}{2} \left( U^{(i)} B_{t} \right)_{k} \right)dB_{k}(t),
\end{align*}
and hence
\begin{align*}
& d \langle X_{t} + A_{i} (t) \rangle= \sum_{k=1}^{m}  \left( \frac{1}{\vert B_{t} \vert} B_{k}(t) + \frac{1}{2} \left( U^{(i)} B_{t} \right)_{k} \right)^{2}dt
\\
& =  \sum_{k=1}^{m}  \frac{B_{k}(t)^{2}}{\vert B_{t} \vert^{2}}dt + \frac{1}{4} \sum_{k=1}^{m} \left( U^{(i)} B_{t}\right)^{2}_{k}dt + \frac{1}{\vert B_{t} \vert} \sum_{k=1}^{m} B_{k}(t) \left( U^{(i)} B_{t}\right)_{k} dt
\\
& = d \langle X \rangle_{t} + d \langle A_{i} \rangle_{t} + \frac{1}{\vert B_{t} \vert } \langle U^{(i)} B_{t}, B_{t} \rangle dt = d \langle X \rangle_{t} + d \langle A_{i} \rangle_{t},
\end{align*}
where in the last line we used that $U^{(i)}$ is skew-symmetric for $i=1,\ldots, n$. Thus, by \cite[Ch.2 Sect Theorem 7.3]{IkedaWatanabe1989} there exists an $(n+1)$-dimensional  Brownian motion $\left(  W_{0} (t), W_{1} (t), \ldots, W_{n} (t) \right)$ such that
\begin{align*}
& X_{t} = W_{0} (\langle X \rangle_{t} ), & A_{1} (t)= W_{1} (\langle A_{1} \rangle_{t} ), && \ldots \ldots &&  A_{n} (t)= W_{n} (\langle A_{n} \rangle_{t} ),
\end{align*}
that is,
\begin{align*}
& X_{t} = W_{0} (t), & A_{1} (t)= W_{1} ( \tau (t) ), && \ldots \ldots &&  A_{n} (t)= W_{n} ( \tau (t)),
\end{align*}
where $ \tau (t) := \frac{1}{4}\int_{0}^{t} \vert B_{s} \vert^{2}ds$. In particular, $X_{t}$ is independent of $W_{t} : =\left( W_{1}(t), \ldots, W_{n} (t) \right)$. The proof is complete once we prove that $W_{t}$ is independent of $\vert B_{t} \vert$. Note that
\[
\vert B_{t} \vert^{2} = 2\int_{0}^{t}  \vert B_{s} \vert dX_{s} + mt,
\]
and hence $\sigma \{ \vert B_{s} \vert,  s\leqslant t \} \subset \sigma \{ X_{s},  s\leqslant t \}$, and thus $\vert B_{t} \vert$ is independent of $W_{t}$.
\end{proof}

\subsection{Spectral bounds}

\begin{theorem}\label{thm.main}
Let $\mathfrak{g}$ be an $H$-type Lie algebra with center $\mathfrak{z}$ and set $n:=\dim \mathfrak{z}$, $m:=\dim \mathfrak{z}^{\perp}$. Let $\mathbb{G}$ be the corresponding $H$-type group with the homogeneous norm $\vert \cdot \vert$ defined by \eqref{eqn.hom.norm.H.typw} and the sub-Laplacian $\Delta_{\mathbb{G}}$. Then
\begin{equation}\label{eqn.main.bound}
\lambda_{1}^{(m)} \leqslant \lambda_{1} \leqslant c \left( \lambda_{1}^{(m)},\lambda_{1}^{(n)}  \right),
\end{equation}
where $\lambda_{1} = \lambda_{1} (m,n)$ is the spectral gap of $-\frac{1}{2} \Delta_{\mathbb{G}}$ restricted to the unit ball $\{ x\in\mathbb{G}: \vert x \vert <1\}$, and
\begin{align*}
& c \left( \lambda_{1}^{(m)},\lambda_{1}^{(n)}  \right):= f( x^{\ast} )= \inf_{0 < x<1} f(x),
\\
& f(x) = \frac{\lambda_{1}^{(m)}}{\sqrt{1-x}} + \frac{\lambda_{1}^{(n)} \sqrt{1-x}}{4x},
\\
& x^{\ast}= \frac{\sqrt{ \left( \lambda_{1}^{(n)}\right)^{2} + 32 \lambda_{1}^{(n)} \lambda_{1}^{(m)} } - 3\lambda_{1}^{(n)}  }{2 \left( 4\lambda_{1}^{(m)} - \lambda_{1}^{(n)} \right)},
\end{align*}
where $\lambda_{1}^{(n)}$ is the lowest Dirichlet eigenvalue on the unit ball in $\R^{n}$ defined in Notation \ref{notation.evalues}.
\end{theorem}

\begin{corollary}
We have that
\begin{equation}\label{eqn.asymptotic}
\lim_{m\rightarrow \infty} \frac{\lambda_{1}(m,n)}{\lambda_{1}^{(m)}}=1,
\end{equation}
that is, when the dimension $m$ of the center of $\mathbb{G}$  is large, the hypoelliptic spectral gap is approximated by the spectral gap on $\R^{m}$. Moreover, for any $m>n$
\begin{align*}
\lambda_{1}^{(m)} \leqslant \lambda_{1}(m,n) \leqslant 2\lambda_{1}^{(m)},
\end{align*}
which gives a bound for small values of $m$ and $n$.
\end{corollary}

\begin{proof}
The lower bound in \eqref{eqn.main.bound}  follows from the small deviation principle \eqref{eqn.small.dev.BM} for an $\R^{m}$-valued Brownian motion and the fact that
\[
\Prob \left( \max_{0 \leqslant t \leqslant 1}\vert g_{t} \vert < \varepsilon  \right) \leqslant\Prob \left( \max_{0 \leqslant t \leqslant 1}\vert B_{t} \vert < \varepsilon  \right).
\]
Let us now prove the upper bound. By Proposition~\ref{prop.MC.for.Brownian.motion} and Lemma~\ref{lemma.IW.general}  it follows that a horizontal Brownian motion $g_{t}$ on $\mathbb{G}$ can be written as $g_{t} = \left( B_{t}, A_{t} \right) = \left( B_{t}, W_{\tau (t)} \right)$, where $B_{t}$ and $W_{t}$ are $m$-dimensional  and $n$-dimensional independent Brownian motions, and $\tau (t) = \frac{1}{4}\int_{0}^{t} \vert B_{s} \vert^{2} ds$. For any $x \in (0, 1)$ we have
\begin{align*}
& \Prob \left( \max_{0 \leqslant t \leqslant 1}\vert g_{t} \vert < \varepsilon \right) = \Prob \left( \max_{0 \leqslant t \leqslant 1}\vert B_{t} \vert^{4}_{\R^{m}} + \vert  A_{t} \vert^{2}_{\R^{n}} < \varepsilon^{4} \right)
\\
& \geqslant \Prob \left( \max_{0 \leqslant t \leqslant 1}\vert B_{t} \vert^{4}_{\R^{m}} < (1-x) \varepsilon^{4}, \; \max_{0 \leqslant t \leqslant 1}\vert A_{t} \vert^{2}_{\R^{n}} < x \varepsilon^{4}  \right)
\\
&  = \Prob \left( \max_{0 \leqslant t \leqslant 1}\vert B_{t} \vert^{4}_{\R^{m}} < (1-x) \varepsilon^{4}, \; \max_{0 \leqslant t \leqslant 1}\vert W_{\tau (t)} \vert^{2}_{\R^{n}} < x \varepsilon^{4}  \right)
\\
&  = \Prob \left( \max_{0 \leqslant t \leqslant 1}\vert B_{t} \vert^{4}_{\R^{m}} < (1-x) \varepsilon^{4}, \; \max_{0 \leqslant t \leqslant \tau (1)}\vert W_{t} \vert^{2}_{\R^{n}} < x \varepsilon^{4}  \right)
\\
& \geqslant \Prob \left( \max_{0 \leqslant t \leqslant 1}\vert B_{t} \vert^{4}_{\R^{m}} < (1-x) \varepsilon^{4}, \; \max_{0 \leqslant t \leqslant  \frac{1}{4} \sqrt{1-x} \varepsilon^{2}} \vert W_{t} \vert^{2}_{\R^{n}} < x \varepsilon^{4}  \right)
\\
&= \Prob \left( \max_{0 \leqslant t \leqslant 1}\vert B_{t} \vert^{4}_{\R^{m}} < (1-x) \varepsilon^{4}, \; \max_{0 \leqslant t \leqslant 1 } \vert W_{t  \frac{1}{4} \sqrt{1-x} \varepsilon^{2}} \vert^{2}_{\R^{n}} < x \varepsilon^{4}  \right)
\\
& =  \Prob \left( \max_{0 \leqslant t \leqslant 1}\vert B_{t} \vert^{4}_{\R^{m}} < (1-x) \varepsilon^{4}, \; \max_{0 \leqslant t \leqslant 1 }  \frac{1}{4} \sqrt{1-x} \varepsilon^{2} \vert W_{t} \vert^{2}_{\R^{n}} < x \varepsilon^{4}  \right)
\\
& =  \Prob \left( \max_{0 \leqslant t \leqslant 1}\vert B_{t} \vert^{4}_{\R^{m}} < (1-x) \varepsilon^{4}, \; \max_{0 \leqslant t \leqslant 1 } \vert W_{t} \vert^{2}_{\R^{n}} < \frac{4 x \varepsilon^{2}}{\sqrt{1-x}}  \right)
\\
&  = \Prob \left( \max_{0 \leqslant t \leqslant 1}\vert B_{t} \vert^{4}_{\R^{m}} < (1-x) \varepsilon^{4} \right) \Prob \left( \max_{0 \leqslant t \leqslant 1 } \vert W_{t} \vert^{2}_{\R^{n}} < \frac{4 x \varepsilon^{2}}{\sqrt{1-x}}  \right).
\end{align*}
Thus,
\begin{align*}
& \log  \Prob \left( \max_{0 \leqslant t \leqslant 1}\vert g_{t} \vert < \varepsilon \right) \geqslant \log  \Prob \left( \max_{0 \leqslant t \leqslant 1}\vert B_{t} \vert_{\R^{m}} < (1-x)^{\frac{1}{4}} \varepsilon \right)
\\
& + \log  \Prob \left( \max_{0 \leqslant t \leqslant 1 } \vert W_{t} \vert_{\R^{n}} < \frac{2 \sqrt{x} \varepsilon}{(1-x)^{\frac{1}{4}}}  \right),
\end{align*}
and hence
\begin{align*}
& -\varepsilon^{2}  \log  \Prob \left( \max_{0 \leqslant t \leqslant 1}\vert g_{t} \vert < \varepsilon \right)
\\
& \leqslant  -\varepsilon^{2} \sqrt{1-x} \log  \Prob \left( \max_{0 \leqslant t \leqslant 1}\vert B_{t} \vert_{\R^{m}} < (1-x)^{\frac{1}{4}} \varepsilon \right)   \frac{1}{\sqrt{1-x}}
\\
& - \varepsilon^{2} \frac{4x}{\sqrt{1-x}} \log \Prob \left( \max_{0 \leqslant t \leqslant 1 } \vert W_{t} \vert_{\R^{n}} < \frac{2 \sqrt{x} \varepsilon}{(1-x)^{\frac{1}{4}}}  \right) \frac{\sqrt{1-x}}{4x}.
\end{align*}
From the small deviation principle \eqref{eqn.small.dev.BM} for a standard Brownian motion applied to $B_{t}$ and $W_{t}$ and the one for a hypoelliptic Brownian motion applied to $g_{t}$ it follows that
\begin{align}\label{eqn.estimate.2}
\lambda_{1}(m,n) \leqslant  \frac{\lambda_{1}^{(m)}}{\sqrt{1-x}} + \frac{\lambda_{1}^{(n)} \sqrt{1-x}}{4x},
\end{align}
for all $x$ in $(0,1)$. Note that
\[
f(x) :=  \frac{\lambda_{1}^{(m)}}{\sqrt{1-x}} + \frac{\lambda_{1}^{(n)} \sqrt{1-x}}{4x} =   \lambda_{1}^{(m)} \left( \frac{1}{\sqrt{1-x}} + \frac{c  \sqrt{1-x}}{4x} \right) >0
\]
for all $x\in (0,1)$, where $c:= \frac{\lambda_{1}^{(n)}  }{\lambda_{1}^{(m)}}$. Note that $f$ always has a local minimum over $(0,1)$ even if we do not rely on the values of the eigenvalues $\lambda_{1}^{(m)}$ and $\lambda_{1}^{(n)}$, and the minimum is achieved at
\begin{align*}
& x^{\ast} =  \frac{\sqrt{ c^{2} + 32 c} - 3c }{2 \left( 4 - c \right)} \in (0,1).
\end{align*}
which gives \eqref{eqn.main.bound}. Moreover,
\begin{align*}
&  \frac{c}{4} \leqslant  \frac{c \sqrt{33} -3c}{8}   \leqslant \frac{\sqrt{ c^{2} + 32 c} - 3c }{8}  \leqslant x^{\ast}  \leqslant \frac{3\sqrt{c} -c}{4-c},
\end{align*}
since $c<1$ for $m>n$. Thus,
\begin{align*}
 \lambda_{1}^{(m)}  \leqslant \lambda_{1}  (m,n) \leqslant f(x^{\ast}) \leqslant  \lambda_{1}^{(m)}  \left( \frac{\sqrt{4-c}}{\sqrt{4-3 \sqrt{c}}} + \frac{\sqrt{4-c}}{2} \right) \leqslant 2 \lambda_{1}^{(m)},
\end{align*}
for any $m>n$. 

Let us now prove \eqref{eqn.asymptotic}. First note that  $\lambda_{1}^{(d)} \rightarrow \infty$ as $d\rightarrow \infty$. Indeed, by \cite{LiYau1983} we have that
\begin{align*}
\lambda_{1}^{(d)}  \geqslant \frac{d}{d+2} 4\pi^{2}\omega_{d}^{-\frac{4}{d}} \rightarrow \infty, \quad \text{as } d\rightarrow \infty,
\end{align*}
where $\omega_{d}$ denotes the volume of the unit $d$-dimensional ball. Then, by  \eqref{eqn.estimate.2} it follows that
\[
1\leqslant \frac{\lambda_{1}(m,n)}{\lambda_{1}^{(m)}} \leqslant  \frac{1}{\sqrt{1-x}} + \frac{\lambda_{1}^{(n)} }{\lambda_{1}^{(m)}} \frac{ \sqrt{1-x}}{4x},
\]
for all $x\in (0,1)$, and the asymptotic \eqref{eqn.asymptotic} then follows by letting $d$ to infinity first and then $x$ to zero. 
\end{proof}

\begin{remark}\label{rmk.any.norm}
Let $\rho$ be any homogeneous norm on $\mathbb{G}$, and $\mu_{1}$ be the spectral gap of $-\frac{1}{2}\Delta_{\mathbb{G}}$ restricted to the ball $\{ x\in \mathbb{G} : \rho (x) <1\}$. Then there exists a constant $c>1$ such that 
\begin{align}
\frac{1}{c^{2}} \lambda_{1} \leqslant \mu_{1} \leqslant c^{2} \lambda_{1}, \label{eqn.bound.any.norm}
\end{align}
where $\lambda_{1} $ is the spectral gap of $-\frac{1}{2} \Delta_{\mathbb{G}}$ restricted to the homogeneous ball $\{ x\in\mathbb{G} : \vert x \vert <1\}$, where $\vert \cdot \vert$ is given by \eqref{eqn.hom.norm.H.typw}. Indeed, by  \cite[Proposition~5.1.4, p.~230]{BonfiglioliLanconelliUguzzoniBook} all homogeneous norms are equivalent, that is, there exists a constant $c>1$ such that 
\[
c^{-1} \vert x\vert \leqslant \rho (x) \leqslant c \vert x \vert, 
\]
where $\vert \cdot \vert$ is given by \eqref{eqn.hom.norm.H.typw}, and hence 
\begin{align*}
& \Prob \left( \max_{0\leqslant t \leqslant 1} \vert g_{t} \vert < c^{-1} \varepsilon \right) \leqslant \Prob \left( \max_{0\leqslant t \leqslant 1} \rho ( g_{t} ) <  \varepsilon \right) \leqslant\Prob \left( \max_{0\leqslant t \leqslant 1} \vert g_{t} \vert < c\varepsilon \right) .
\end{align*}
Then \eqref{eqn.bound.any.norm} follows by \cite[Corollary 5.4]{CarfagniniGordina2023}.
\end{remark}

\bibliographystyle{amsplain}

\end{document}